\documentclass{amsart}
\usepackage{amssymb}
\usepackage{amsmath}
\usepackage{graphicx}
\usepackage{amscd}
\usepackage{amsthm}
\usepackage{amsfonts}
\usepackage{hyperref}
\usepackage{fancyhdr}
\newtheorem{proposition}{Proposition}[section]
\newtheorem{definition}[proposition]{Definition}
\newtheorem{lemma}[proposition]{Lemma}
\newtheorem{theorem}[proposition]{Theorem}
\newtheorem{corollary}[proposition]{Corollary}

\newcommand\eps{\epsilon}
\newcommand{\comment}[1]{}

\newcommand{\R}{\mathbb{R}}
\newcommand{\D}{\mathbb{D}}
\newcommand{\Z}{\mathbb{Z}}

\newcommand{\csubset}{\subset \subset}
\newcommand{\mass}{\mathbf{mass}}
\newcommand \restrict[2]{\left. #1\right|_{#2}} 
\newcommand \norm[1]{\parallel #1\parallel} 
\newcommand \lcur{[\![} 
\newcommand\rcur{]\!]}
\newcommand \cur[1]{\lcur #1\rcur}
\newcommand\piinv{\pi^{-1}}

\DeclareMathOperator \length{length}
\DeclareMathOperator \diam{diam}

\DeclareMathOperator \area{area}

\DeclareMathOperator \interior{interior}
\DeclareMathOperator \core{core}
\DeclareMathOperator \crust{crust}
\DeclareMathOperator \support{support}

\begin{document}

\title{Piecewise linear approximation of smooth functions of two variables}
\author{Joseph H.G. Fu and Ryan C. Scott}
\thanks{Fu partially supported by NSF grant DMS-1007580. Scott supported by an NSF VIGRE Postdoctoral Fellowship.}
 \address{ Department of Mathematics, 
University of Georgia, 
Athens, GA 30602, USA}

\date{\today}
\lhead{J.H.G. Fu and R.C. Scott}
\rhead{PL Approximation}
\maketitle
\begin{abstract} Given a piecewise linear (PL) function $p$ defined on an open subset of $\R^n$, one may construct by elementary means a unique polyhedron with multiplicities $\D(p)$ in the cotangent bundle $\R^n\times \R^{n*}$ representing the graph of the differential of $p$. Restricting to dimension 2, we show that any smooth function $f(x,y)$ may be approximated by a sequence $p_1,p_2,\dots$ of PL functions such that the areas of the $\D(p_i)$ are locally dominated by the area of the graph of $df$ times a universal constant.
\end{abstract}
\section{Introduction} 
\subsection{General introduction}
In approximating smooth submanifolds by polyhedra, it is well known that pathologies may occur if the triangles of the polyhedra become too thin. For example, the famous   ``Schwartz lantern" (cf. \cite{morvan}, section 3.1.3) shows that the areas of a sequence of  polyhedra $P_i$ approximating a surface $S$ may fail to converge unless the  ratio of the area of a general constituent triangle to the square of its diameter is bounded away from zero. On the other hand, if this ``fatness" condition is imposed then the sequence enjoys convergence not only to first order  (i.e. convergence of areas),  but also to second order in the sense that  certain curvature integrals  (or, more precisely, their polyhedral analogues) for the $P_i$ converge weakly to the corresponding integrals for $S$ (\cite{cms,fu93}).
The key observation is that the corresponding {\it absolute} curvature integrals of the $P_i$ remain bounded in the course of the approximation process. 

However, there is a further subtlety: even if the fatness condition is imposed, the bound on the absolute curvature integrals  of the $P_i$ depends on the $C^2$ norm of $S$, which may be much larger than its  absolute curvature integral (cf. section \ref{sect:quadratic} below). 

In the present paper we show that by means of a more careful approach, the $P_i$ may be chosen so that their absolute curvature integrals are bounded by universal constant multiples of the corresponding integrals for $S$.
In order to isolate the analytic aspects of the process, instead of a smooth surface $S$ we consider  smooth and piecewise linear (PL) {\it functions} $f$ defined on an open subset $U \subset \R^n$, replacing curvatures of $S$ by invariants of the Hessian of $f$.

\subsection{Statement of the main theorem} The curvature integrals of $S$ that we consider are artifacts of the manifold $N(S) \subset \R^3 \times S^2$ of unit normals to $S$. Using the area formula, it is easy to see  that the area of $N(S)$ is bounded above and below by constant multiples of $\int_S 1 + \norm{II} + |\det II |$, where $II$ is the second fundamental form. By the same token, the area of the graph of the differential $df$ of a $C^2$ function is bounded above and below by constant multiples of $\int 1  + \norm {H_f} + |\det H_f|$, where $H_f$ is the Hessian (cf. \eqref{eq:mass formula}).

For certain classes of singular surfaces $S$ and functions $f$, it is known that there exist integral currents $N(S)$ and $\D(f)$ that serve, respectively, as suitable replacements for the manifold of unit normals and the graph of the differential. We state the relevant theorem in the latter case.

\begin{theorem} [\cite{fu89, jerrard2}]\label{thm:ma}
Let $U \subset \R^n$ be an open set and let $f$ be a function with locally integrable differential $df$ (i.e. $f$ belongs to the Sobolev space $ W^{1,1}_{loc}(U)$). Then there is at most one
integral current $\D(f)$ of dimension $n$ in the cotangent bundle $T^* U$, with the properties
\begin{itemize}
\item $\partial \D(f) = 0$
\item $\D(f)$ is {\bf Lagrangian}, in the sense that $\D(f) \llcorner \omega = 0 $, where $\omega $ is the canonical symplectic form of $T^*U$
\item if $K \csubset U$ then $\mass ( \D(f) \llcorner \pi^{-1}K) <\infty$
\item  for any compactly supported continuous function $\phi:T^*U \to \R$ 
\begin{equation}\label{eq:support}
\D (f)( \phi \, \pi^*(dx_1\wedge\dots\wedge dx_n))= \int_{\R^n} \phi(x, df_x) \, d\mathcal L^nx
\end{equation}
\end{itemize}
\end{theorem}
Here $\pi:T^*U \to U$ is the projection. If $\D(f)$ exists then we say that $f$ is a {\it Monge-Amp\`ere (MA) function}, and refer to $\D(f)$ as the {\bf gradient cycle} of $f$.


If $f \in C^{1,1}(U)$ (i.e. $f \in C^1$ and its derivative is locally Lipschitz) then $\D(f)$ is given by integration over the graph of the differential $df$, which is a Lipschitz submanifold of $T^*U$ with orientation induced by the standard orientation of $U \subset \R^n$. Thus if $f$ is not smooth then we may think of $\D(f)$ as a substitute for this graph, and its existence is an indication that $f$ has good second order properties. It is not difficult (cf. \cite{jerrard1}) to construct $\D(p)$ directly for a PL function $p$; for completeness, and to fix ideas, we give the construction in detail for $n=2$ in section 
\ref{sect:d of pl} below.
In these terms our main theorem is the following. We denote the space of PL functions on an open set $V$ by $PL(V)$.
\begin{theorem} \label{thm:main} Let $U \subset \R^2$ be open and $f \in C^{1,1}(U)$. Given an open set $V\csubset U$, there exists a sequence $p_i \in PL(V)$, with $p_i \to \restrict f V$ uniformly, such that
 \begin{equation}\label{eq:main}
\limsup_{i\to \infty} \mass (\D(\restrict{p_i}V )) <\int_V 1 + 2 \sqrt 2 \norm {H_f} + \left|\det H_f\right|\, dA
\end{equation}
\end{theorem}

In fact it is enough to prove Theorem \ref{thm:main} in the case that $f \in C^2(U)$: convolving with a mollifier we may approximate any $f \in C^{1,1}(U)$ locally by $C^2$ functions $f_i$, in such a way that the Hessians of the $f_i$ converge pointwise a.e. to that of $f$. Invoking the estimates \eqref{eq:main} for the $f_i$ and using a diagonal process, the desired estimate for $f$ follows from the dominated convergence theorem.

As mentioned above, there is a weaker form of Theorem \ref{thm:main} that holds for any PL approximation in which the underlying triangulations are uniformly fat. This is given here as Proposition \ref{prop:crude}. In this result the masses of the
$\D(\restrict {p_i} V)$ are bounded in terms of the integral of the square of the norm of the Hessian matrix $H_f$.  By the area formula,
\begin{align}\label{eq:mass formula}
\mass~\D(\restrict f V) &= \int_V   \left(1 + \norm{H_f}^2+ (\det H_f)^2\right)^{\frac 1 2}~dx~dy\\
\notag & \simeq \int_V 1 + \norm{H_f} +\left|\det H_f\right|~dx ~dy
\end{align}
for $f \in C^2(U)$, where 
$$
\norm{H_f} = \left[\left(\frac{\partial^2f}{\partial x^2}\right)^2+ \left(\frac{\partial^2f}{\partial y^2}\right)^2 + 2\left( \frac{\partial^2f}{\partial x\partial y}\right)^2 \right]^{\frac 1 2}
$$
is the Hilbert-Schmidt norm of the Hessian. Thus the bound of Theorem \ref{thm:main} is  stronger in the event that $H_f$ has one small and one large eigenvalue over a large region of $U$.

\subsection{Some consequences}
All functions $f$ known to be MA are {\it strongly approximable} in the sense of \cite{fu11}, i.e. there is a sequence $f_j$ of $C^2$ functions that converge to $f$ in $L^1_{loc}$ such that the $\D(f_j)$ are locally uniformly bounded in mass. The differential current $\D(f)$ then arises as the limit of the $\D(f_j)$: the Federer-Fleming compactness theorem for integral currents implies that a subsequence converges, and the uniqueness result from the first paragraph above implies that the limit is uniquely determined. A fundamental conjecture states that every MA function is strongly approximable.

In view of the simple construction of $\D(f)$ for PL functions, it is natural to consider also the class of {\bf PL strongly approximable} functions, which are defined analogously by taking the approximating functions $f_j$ to be PL instead of $C^2$. Again one conjectures that any MA function is PL strongly approximable. Theorem \ref{thm:main} implies that these two notions coincide in dimension $n=2$.
\begin{corollary} Let $U\subset \R^2$ and $f \in W^{1,1}_{loc}(U)$. Then $f$ is strongly approximable iff $f$ is PL strongly approximable.
\end{corollary}
That a PL strongly approximable function must be strongly approximable follows from a construction of U. Brehm and W. K\"uhnel \cite{bk}, which (in our language) shows that any PL function $p$ of two variables may be approximated by $C^2$ functions $f_j$ in such a way that the local masses of the $\D(f_j)$ are bounded in terms of the local mass of  $\D(p)$. The other direction follows from Theorem \ref{thm:main}.

\subsection{Outline of the paper}  The basic idea of the proof of Theorem \ref{thm:main} arises from the analysis   of PL approximations to a given quadratic form $q(x,y)$ given in section \ref{sect:example}: we interpolate linearly between the values of $q$ on the vertices of a square grid to obtain a PL function $p$, and show that if the grid is aligned with the eigenvectors of $q$ then the mass of $\D(p)$ over an open subset $ U \subset \R^2$ is roughly proportional to $\int_U 1 + \norm q +\left|\det q\right|$. If, on the other hand, the grid is aligned haphazardly, then this mass is more like $\int_V 1 + \norm q +\norm q^2$, corresponding to the estimate of Proposition \ref{prop:crude}.

In order to carry this analysis over to a given $C^2$ function $f$ defined in a neighborhood of a set $V$, we attempt to impose a subdivided square grid of small mesh size so as to align with the eigenvectors of $H_f$. Since the eigenvectors are not constant this is obviously not possible, so instead we do this in approximate fashion in open sets where the oscillation of $H_f$ is small, the union of which yields a triangulated set $W$ such that $V-W$ has small area. We also take care to ensure that the square grids remain separated from each other by a distance proportional to the mesh size. This allows us to invoke a well-known (albeit apparently unpublished) result  from the theory of mesh generation, stating that this partial triangulation of $V$ may be completed to a full triangulation with a lower bound on the fatness of the triangles:
\begin{theorem}[\cite{chew}, p. 10]\label{thm:chew} Let $h>0$, and let $P \subset \R^2$ be a closed  polygonal region, and $T$ a system of vertices and edges in $P$, such that no two vertices of $ T$ are closer than $h$ and  every edge of $T$ has length  $ \in [h, h\sqrt 3]$. Then there exists a triangulation of $P$ that includes $T$ such that all edges have length  $\in [h,2h]$ and all angles $\in [\frac \pi 6, \frac {2\pi} 3]$.
\end{theorem}
 Our estimates in the quadratic case carry over to yield the desired bounds on the mass of $\D(\restrict p W)$. Meanwhile the mass of $\D(\restrict p {V-W})$ is controlled only by the cruder bounds of Proposition \ref{prop:crude}, but since $V-W$ has small measure this quantity is also small.

 The absence of the analogue of Theorem \ref{thm:chew} in dimensions larger than 2 appears to constitute the principal obstacle to extending Theorem \ref{thm:main} to  higher dimensions.

For simplicity we will identify the dual space $\R^{2*}$ with $\R^2$ in the usual way, and conflate the differential $df$ and the gradient $\nabla f$ of a function $f$.

\subsection {Acknowledgements} We would like to thank Bob Jerrard for interesting and helpful discussions. 

\section{The gradient cycle of a PL function}\label{sect:example} 

A {\bf PL function}  on an open subset $W \subset \R^n$ is a continuous function $p$ on $W$, for which there exists a triangulation $T$ of a polyhedral set $W'\supset W$ such that the restriction of $p$ to every simplex of $T$ is affine. Such a $T$ is an {\bf underlying triangulation} for $p$. The space of all such functions is denoted $PL(W)$.

Given an oriented simplex $\phi$ of dimension $k$, we denote by $\cur \phi$ the $k$-dimensional current given by integration over $\phi$. In particular, if $\phi$ is a point (i.e. $k=0$) then $\cur \phi$ is simply evaluation at $\phi$. Such currents are clearly {\it integral currents} in the sense of \cite{gmt}, Chapter 4. The space of integral currents of dimension $k$ on a smooth manifold $M$ is denoted $\mathbb I_k(M)$.

\subsection{Construction of $\D(p)$ for $ p \in PL$} \label{sect:d of pl} We discuss only the case $n=2$.  Given a PL function $p$ defined on an open set $W \subset \R^2$, let $T$ be an underlying triangulation for $p$ consisting of relatively open simplices $\tau_i, \sigma_j, \rho_k$ of dimensions $ 2,1,0$ respectively.

Put $\nabla_i$ for the (constant) value of $dp$ on $\tau_i$, and set
 $$
 D_2 : = \sum_i \cur {\tau_i} \times \cur{\nabla_i}\in \mathbb I_2(W\times \R^2)
 $$ 
 where the orientations  of the $\tau_i$ are induced by the standard orientation of $\R^2$. Thus
  \begin{align*}
\partial D_2 &=
 \sum_i \partial \cur {\tau_i} \times \cur{\nabla_i}\\
 &= \sum_j \cur {\sigma_j} \times\left( \cur{\nabla_j^+ }- \cur {\nabla_j^-}\right)
 \end{align*}
with appropriate orientations for the edges $\sigma_j$, where $\nabla_j^+,\nabla_j^-$ are the respective values of the gradient of $p$ on the neighboring triangles lying to the left and to the right of $\sigma_j$ with respect to the given  orientation. 

Put $s_j$ for the oriented line segment from $\nabla_j^- $ to $\nabla _j^+$. Since the affine functions giving $p$ on the neighboring triangles agree along $\sigma_j$, it follows that $s_j$ is perpendicular to $\sigma_j$.
This implies that the current
$$
D_1 := -\sum_j \, \cur {\sigma_j} \times \cur {s_j}
$$
is Lagrangian, with
\begin{equation}
\partial(D_1+ D_2) =- \sum_j\, \partial \cur{\sigma_j }\times \cur {s_j} 
\end{equation}

Given a vertex $\rho_k$, let $\tau_1,\dots,\tau_{N-1},\tau_N= \tau_0$ be the triangles of $T$ having $\rho_k$ as a vertex, listed in counterclockwise order, with corresponding gradients $\nabla_1,\dots, \nabla_N= \nabla_0$. Let $\sigma_i$ be the edge between $\tau_{i-1}$ and $\tau_i$, $ i = 1,\dots,N$, oriented to point away from $\rho_k$, and $s_i$ the oriented line segment $\overline{\nabla_{i-1},\nabla_i}$. Then the part of $\partial(D_1 + D_2)$ lying above $\rho_k$ is $\cur{\rho_k}\times\sum_{i=1}^N   \cur{s_i}$. Put $P_k\subset \R^2$ for the unique bounded polygonal region-with-multiplicities with $\partial \cur{P_k } = - \sum \cur{s_i}$. 

Finally, put
$D_0 := \sum_k \cur{\rho_k} \times \cur{P_k}$. Then $D_0 + D_1 + D_2$ satisfies the hypothesis of Theorem \ref{thm:ma} for $f= p$, so $\D(p) = D_0 + D_1 + D_2$.

We remark that if $\gamma$ is any closed oriented rectifiable curve in $\R^2$ then the unique compactly supported current $T$ of dimension 2 with $\partial T = \gamma$ may be constructed as follows. Choose an arbitrary point $x_0 \in \R^2$, and define the map $F: \R \times \gamma \to \R^2$ by
$$
F(t, x) := (1-t)x_0 + t x.
$$
Then $T= F_*(\cur{0,1} \times \gamma)$, i.e. the {\it join} of $\gamma$ with $x_0$.

\subsection{PL approximation of a quadratic function}\label{sect:quadratic}

\begin{definition} \label{def:fT} Let $T$ be a finite triangulation of a polygonal planar region $W$. Given a function $f$ defined over $W$, put $f_T$ for the piecewise linear function on $W$ that agrees with $f$ at the vertices of $T$ and is affine on each triangle of $T$.
\end{definition}

\begin{definition}\label{def:complexes}
Denote by $T^*$ the triangulation of $\R^2$ with vertices given by the points of the integer lattice $\Z \times \Z$ and edges given by the sides of the squares of the lattice, together with their diagonals of slope $-1$. 
By a {\bf square mesh of size $\eps$} we mean a finite triangulation $T$ 
 that is congruent to a subcomplex of $\eps T^*$. 

The {\bf axes} of  $T$ are the images of the $x$- and $y$-axes of $T^*$ under the congruence map. 

The {\bf support} of $T$ is the union of the simplices comprising $T$.

The {\bf interior} of $T$ is the interior of the support of $T$.

The {\bf core} of $T$ is the interior of the union of the closed squares of side $\eps$ centered at the vertices of $T$ that lie in $\interior T$.

The {\bf crust} of $T$ is the union of the closed squares that intersect the boundary of $\support T$.
\end{definition}

\begin{figure}[h]
\begin{center}
\setlength{\unitlength}{240pt}
\begin{picture}(1,.76)
    \put(0,0){\includegraphics[width=\unitlength]{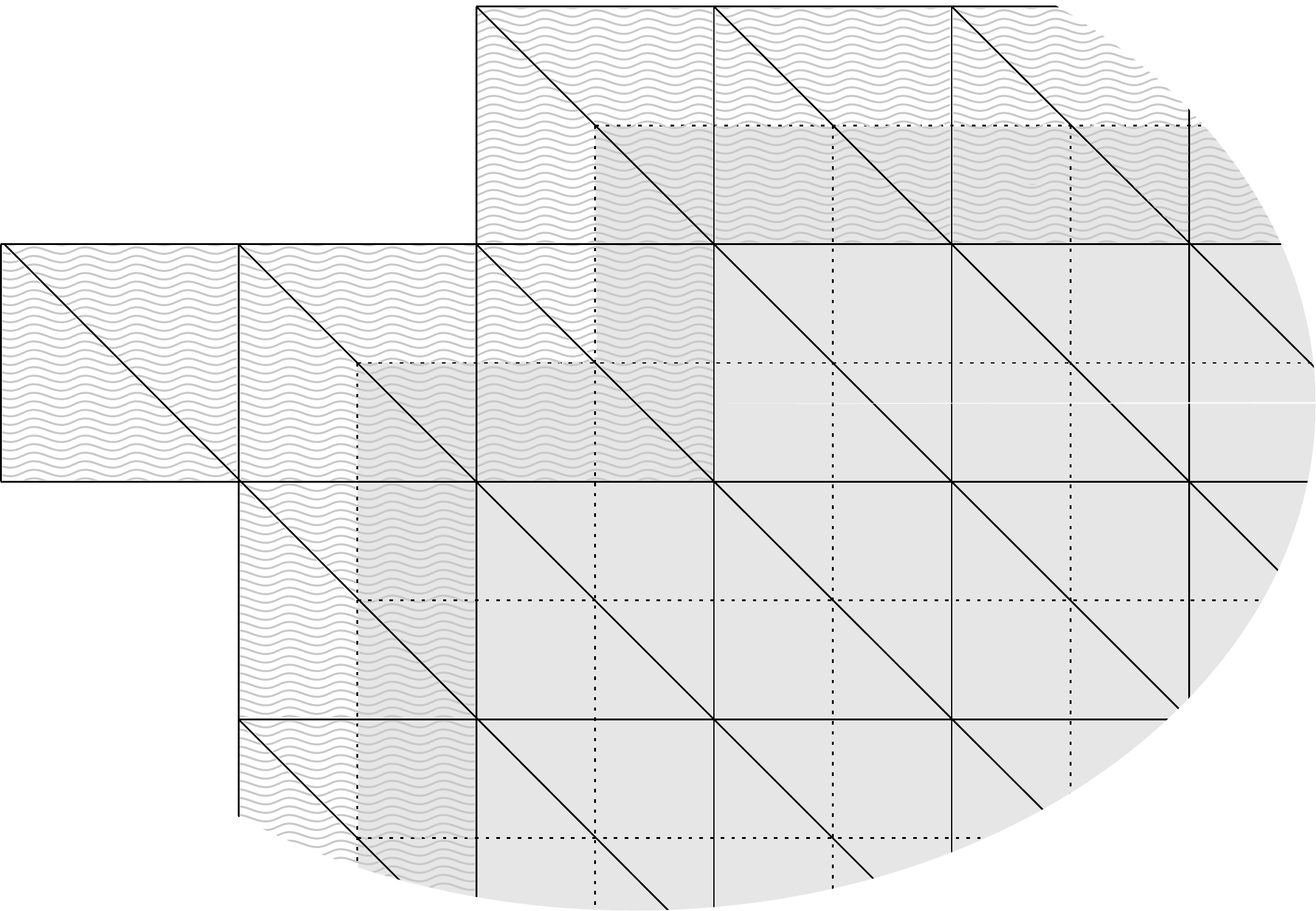}}  
  \end{picture}
\end{center}
\caption{The crust and core of part of a square mesh. The wavy region is the crust; the grey region is the core.} 
\label{fig:crust}
\end{figure}

The proof of Theorem \ref{thm:main} is based on the analysis of the following situation. Given $a,b,c \in \R$ we put
\[
q(x,y):=ax^2 + 2bxy + cy^2, \quad p:= q_{T^*}.
\]
  The differential current $\D(p) = D_0 + D_1 + D_2$ may be constructed as described in section \ref{sect:d of pl}, where each $D_i$ lives over the $i$-skeleton of $T^*$. In particular, $D_0$ is given by integration over $ \bigcup_{(m,n) \in \Z \times \Z} \{(m,n)\}\}\times P_{m,n}$, where $P_{m,n}\subset \R^{2*}$ is a polygonal region with multiplicities.  It is clear that the $P_{m,n}$ are translates of $P:= P_{0,0}$.

\begin{figure}[h]
\begin{center}
\setlength{\unitlength}{300pt}
\begin{picture}(1,0.7693005)
    \put(0,0){\includegraphics[width=\unitlength]{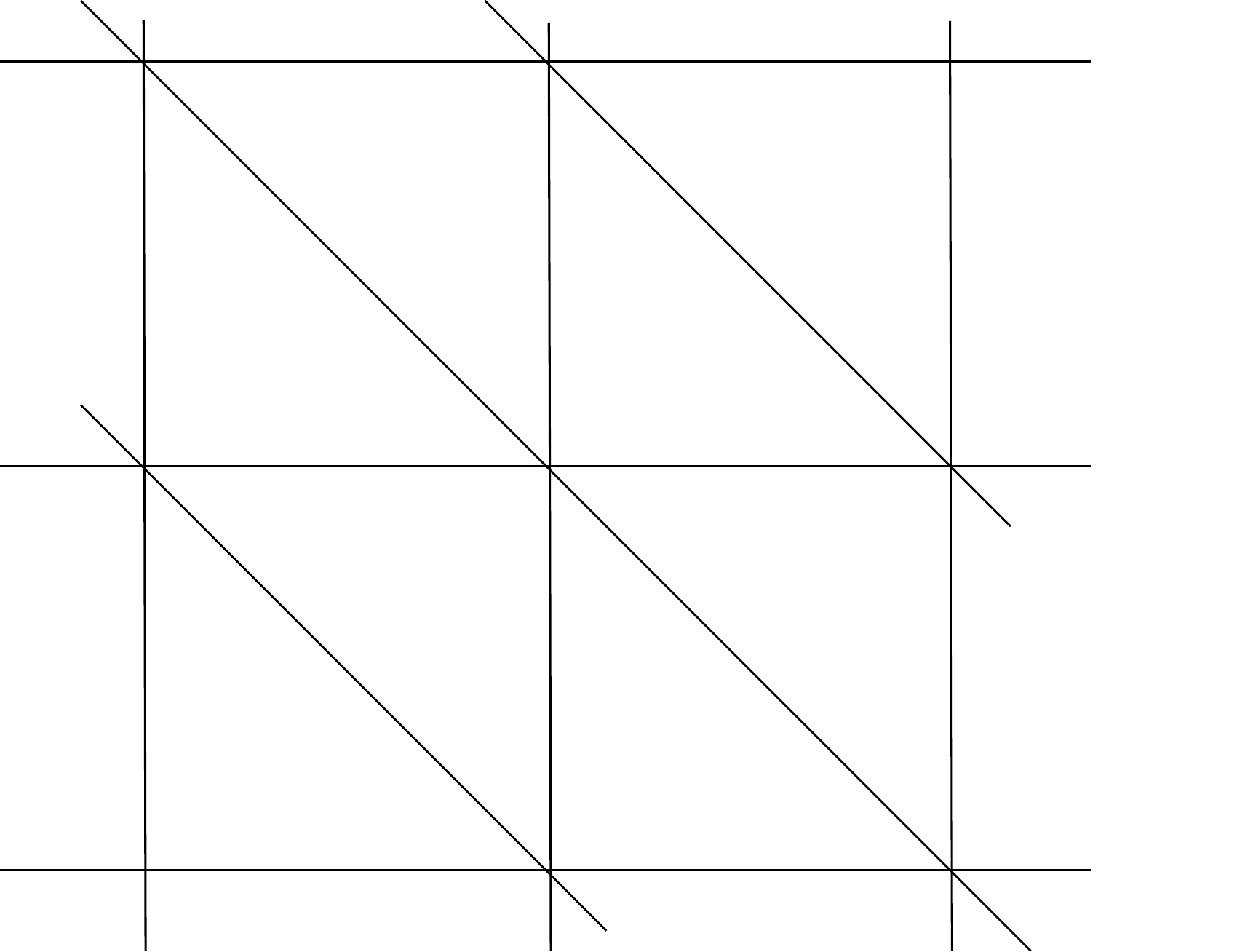}}
    \put(0.50663618,0.52356205){$\tau_1$}
    \put(0.32686205,0.62162065){$\tau_2$}
    \put(0.17977414,0.52356205){$\tau_3$}
    \put(0.32686205,0.27841546){$\tau_4$}
    \put(0.52297929,0.16401378){$\tau_5$}
    \put(0.65372413,0.27841546){$\tau_6$}
    \put(0.45760688,0.40916032){$(0,0)$}
    \put(0.45760688,0.73602238){$(0,1)$}
    \put(0.13074482,0.73602238){$(-1,1)$}
    \put(0.13074482,0.40916032){$(-1,0)$}
    \put(0.45760688,0.08229826){$(0,-1)$}
    \put(0.78446894,0.40916032){$(1,0)$}
    \put(0.78446894,0.08229826){$(1,-1)$}
  \end{picture}
\end{center}
\caption{The neighborhood of a vertex of $T^*$} 
\label{fig:grid}
\end{figure}
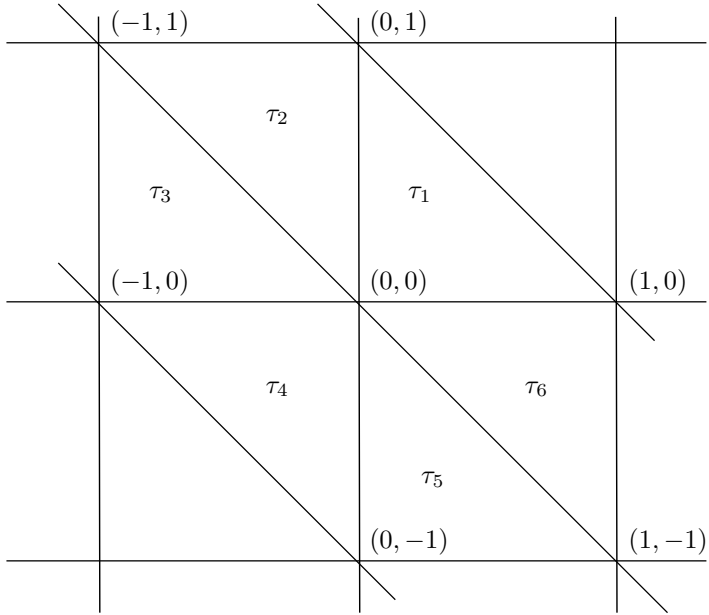

  We wish to compute the mass of  $P$ in terms of $a,b,c$.  We label the six triangles surrounding the origin by $\tau_1,\ldots,\tau_6$ as in Figure \ref{fig:grid}. Denoting by $\nabla_i$ the gradient of $\restrict p{\tau_i}, i = 1,\dots, 6$, 
  we compute
\begin{align}
\notag\nabla_1&=(a,c) \\
\notag\nabla_2& =(-a+2b,c) \\
\label{eq:gradients}\nabla_3&=(-a,-2b+c)\\
\notag\nabla_4&=(-a,-c)\\
\notag\nabla_5&=(a-2b,-c)\\
\notag\nabla_6&=(a,2b-c) .
\end{align} 
The boundary of the region $P$ is then the closed hexagonal path with the $\nabla_i$ as successive vertices.

In particular, if  $b=0$ then  $P$ degenerates to a rectangle with signed area  equal to $4ac =\det H_q$. Furthermore, in this case the support of $D_1$ lies above the horizontal and vertical segments of the mesh. The segments in the fiber of $D_1$ over the vertical segments have length $2|a|$, and over the horizontal segments have length $2|c|$.
 From these observations we easily draw the following conclusion.
\begin{lemma}\label{lem:aligned} Suppose $q(x,y)= ax^2 + 2bxy +cy^2 + \alpha x +\beta y + \gamma$ is a quadratic function and $T$ is a square mesh  of size $\eps$, such that the axes of  $T$ are parallel to the eigenvectors of $H_q$. Then
\begin{equation}\label{eq:mass quad}
\mass (\D(\restrict{q_T}{\core T})) \le \left( 1 + 2\sqrt 2 \norm {H_q} + |\det H_q| \right)\area({\core  T})
\end{equation}
\end{lemma}
\begin{proof} The addition of the linear and constant terms has no effect on the calculations above, since they yield merely a translation of the gradient cycle by the vector $(\alpha,\beta)$. Taking the axes parallel to the eigenvectors is equivalent to taking $b=0$. Scaling by $\eps$, the relation \eqref{eq:mass quad} clearly holds if $\core T$ is replaced by any of its constituent squares. 
\end{proof}

On the other hand, if the axes of $T$ are not aligned with the eigenvectors of $q$, then the masses of the $P=D_0(q,T,\rho)$ may be much larger. For example, in the case of  the degenerate quadratic form
$q(x,y) = x^2 -2xy + y^2$ we find that $P $ is as shown in Figure \ref{fig:hex}. The mass of $P$ is then equal to $8$, while $\det H_q =0$, illustrating  the fact that  the  bound of Theorem \ref{thm:main} cannot be achieved by a haphazard triangulation (as in Proposition \ref{prop:crude} below).

\begin{figure}
\begin{center}
\setlength{\unitlength}{100pt}
\begin{picture}(1,1)
    \put(0,0){\includegraphics[width=\unitlength]{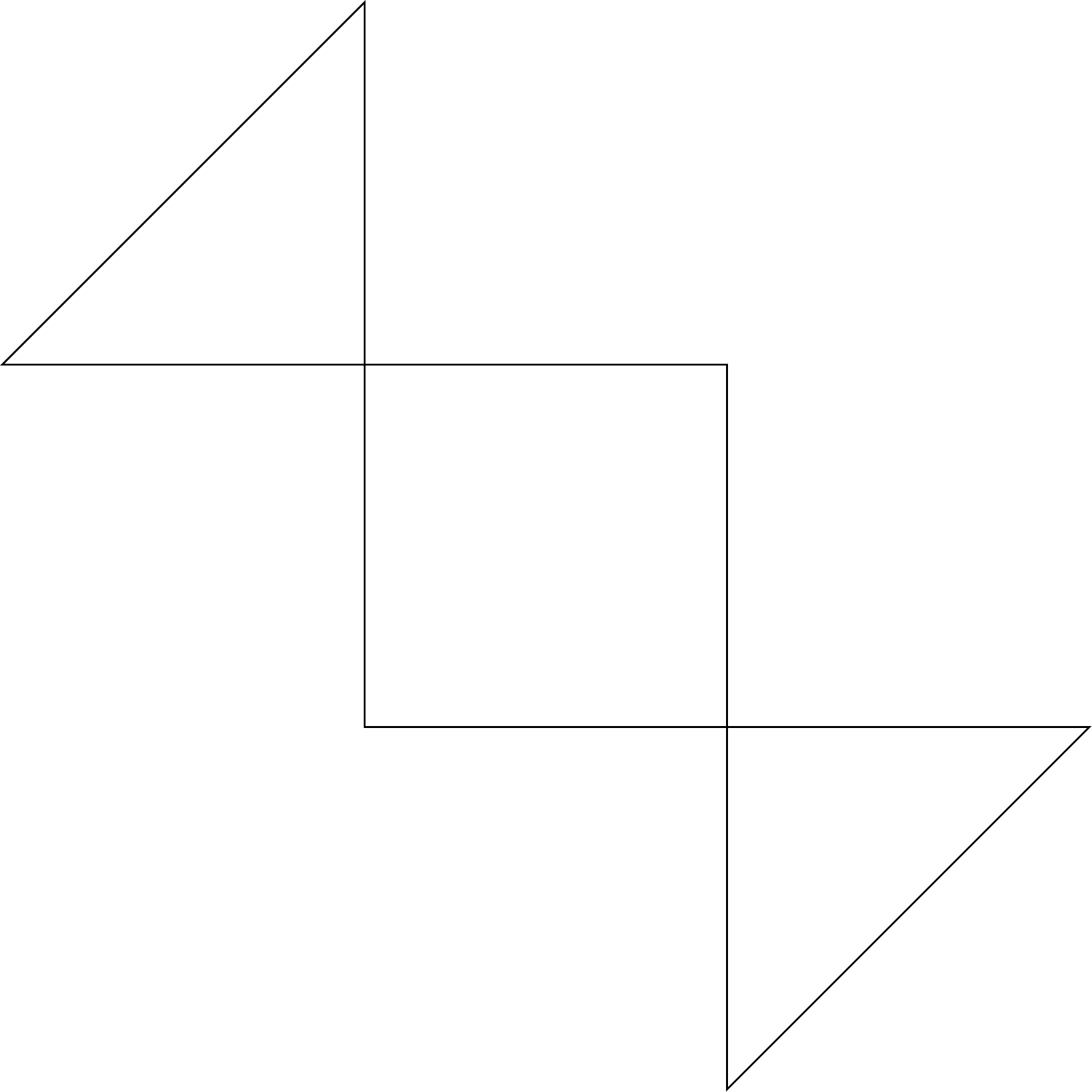}}
    \put(0.7,0.6){$\nabla_1=(1,1)$}
  
  \end{picture}
\end{center}
\caption{The image of a badly aligned triangulation} 
\label{fig:hex}
\end{figure}

Although we will not need this fact, it is interesting to note that the {\it algebraic} area of  $P$ is always precisely equal to $\det H_q$.
 
\subsection{Perturbations}\label{sect:perturb}

For a given triangle $\tau$, vertex $\rho$ and edge $\sigma$ of $T$ , we put 
\begin{align*}
D_0(f,T,\rho)&:=\D(f_T)\llcorner \piinv \rho\\
D_1(f,T,\sigma)&:= \D(f_T)\llcorner \piinv \sigma\\
\nabla_\tau f &:= \nabla(\restrict{ f_T} \tau)
\end{align*}
and 
$$
\Delta_0(f,T,\rho):= \max | \nabla_{\tau_1} { f}-\nabla_{\tau_2} { f}|
$$ 
where $\tau_1, \tau_2$ range over all triangles of $T$ incident to $\rho$, i.e. $\Delta_0(f,T,\rho)$ is the diameter of the polygon $P$ lying above $\rho$.

\begin{lemma}\label{pert} Let $T$ be a square mesh of size $\eps$ and $f,g$ functions defined on the support of $T$. Then
\begin{enumerate}
\item\label{1st conclusion} for any triangle  $\tau$  of $T$,
$$|\nabla_\tau f -\nabla_\tau g| \le4\eps^{-1} \sup|f-g|$$
\item for any interior edge $\sigma$  of $T$, 
$$
\left|\mass\,D_1(f,T,\sigma) - \mass\, D_1(g,T,\sigma) \right| \le 8 \sqrt 2\sup|f-g|
$$
\item  for any interior vertex $\rho$ of $T$, 
$$
\mass\,\left(D_0(f,T,\rho)-D_0(g,T,\rho)\right)\le 48\eps^{-1} \sup |f-g| \Delta_0(g,T,\rho)+ 96 \eps^{-2} \sup|f-g|^2
$$
\end{enumerate}
\end{lemma}
\begin{proof} We may assume that the axes of $T$ are the $x$- and $y$-axes.

(1):  Let $\tau$ be a triangle of $T$ with vertices $(x_0,y_0), (x_0 \pm \eps, y_0), (x_0, y_0\pm \eps)$. Then
$$
\nabla_\tau f = \pm \eps^{-1}(f(x_0 \pm \eps, y_0) - f(x_0,y_0), f(x_0, y_0\pm \eps)-f(x_0,y_0))
$$
from which the conclusion follows at once.

(2): By conclusion (1), the fibers of $D_1(f,T,\sigma), D_1(g,T,\sigma)$ are line segments whose endpoints differ by at most $4\eps^{-1} \sup|{f-g}|$. Thus the lengths of these segments differ by at most twice that, and the conclusion follows since $\length \sigma \le \eps \sqrt 2$.

(3):  Given any point $A \in \R^2$, the current $D_0(f,T,\rho)$ may be expressed as the join of $A$ with the bounding polygon. This may in turn may be expressed as the sum of the joins of $A$ with each of the six edges. We take $A:=\nabla_\tau f$ for some arbitrarily chosen triangle $\tau$ incident to $\rho$. For any pair of adjacent triangles $\tau_1,\tau_2$ incident to $\rho$, put
$$
B:= \nabla_{\tau_1}f ,\ C:=\nabla_{\tau_2}f, \ B':= \nabla_{\tau_1}g ,\ C':=\nabla_{\tau_2}g
$$
We claim that 
$$
\mass(ABC- AB'C') \le 8\eps^{-1} \sup |f-g| \Delta_0(f,T,\rho)+ 16 \eps^{-2} \sup|f-g|^2
$$
This follows from the relations
\begin{align*}
\mass(ABC -ABC') &\le 4\eps^{-1} \sup |f-g| \Delta_0(f,T,\rho)\\
\mass(ABC' -AB'C') &\le  4\eps^{-1} \sup |f-g|\left(\Delta_0(f,T,\rho) +4\eps^{-1} \sup |f-g|\right)
\end{align*}
which follow in turn from conclusion \eqref{1st conclusion} and Lemma \ref{lem:triangles} below.
\end{proof}

\begin{lemma} \label{lem:triangles}Let $P,Q,R,R' \in \R^2$. Put $\Delta, \Delta'$ for the oriented triangles $PQR, PQR'$ respectively. Then
$$
\mass ( \Delta- \Delta') \le |R-R'| \diam  \Delta.
$$
\end{lemma}
\begin{proof} The mass of $ \Delta- \Delta'$ is the area of the symmetric difference of the two triangles. This symmetric difference is included in the union of the triangles $PR R', QR R'$, with total area $\le \frac 1 2  (|P-R|+|Q-R|) |R-R'|$.
\end{proof}

\begin{lemma}\label{lem:taylor} Let $T$ be a square mesh of size $\eps>0$ such that $\interior T \subset B(x_0,r_0)$. Let $f \in C^2(B(x_0,r_0))$. Suppose that the axes of $T$ are parallel to the eigenvectors of $H_f(x_0)$, and that
 $ \norm{H_f(x) -H_f(x_0)}<\omega$ for $x \in B(x_0,r_0)$. Then
\begin{align*}
\mass (\D(\restrict{f_T}{\core T})) < &\int_{\core T} 1 + 2\sqrt 2 \norm{ H_f} + |\det H_f | \, dA\\
&+ \area (\core T)\times \\
&\quad \left[12\sqrt 2 \eps^{-2} \omega r_0^2 + 24  \norm {H_f(x_0)} \eps^{-2}\omega r_0^2 + 24 \eps^{-4} \omega^2 r_0^4\right.\\
& \quad \quad \quad \left.+ 2\omega\sqrt 2  + 4\norm{H_f(x_0)}\omega +2\omega^2\right]
\end{align*}
\end{lemma} 
\begin{proof} Put $q$ for the 2nd order Taylor polynomial of $f$ at $x_0$, so that $H_q\equiv H_f(x_0)$. Then 
$$
\sup_U |f-q| < \frac {\omega   r_0^2}2
$$
There are $\eps^{-2} \area(\core T)$ constituent squares  in $\core T$, each of which contains one vertex and six half-edges of $T$. Referring to \eqref{eq:gradients}, it is clear that
$$
\Delta_0(q,T,\rho) \le \norm{H_q} \eps
$$
for every  vertex $\rho$ lying in $\core T$.
 Thus by Lemmas \ref{lem:aligned} and \ref{pert}, 
\begin{align}
\notag\mass(\D(\restrict {f_T}{\core T})) \le &\mass (\D(\restrict {q_T}{\core T})) \\
\notag &+ \area (\core T)\left[12\sqrt 2 \eps^{-2} \omega r_0^2 + 24  \norm {H_q} \eps^{-2}\omega r_0^2 + 24 \eps^{-4} \omega^2 r_0^4\right]\\
\le &\area({\core T})\times\\
\notag& \left( 1 + 2\sqrt 2 \norm {H_q} + |\det H_q| +\right.\\
\notag&
\left.12\sqrt 2 \eps^{-2} \omega r_0^2 + 24  \norm {H_q} \eps^{-2}\omega r_0^2 + 24 \eps^{-4} \omega^2 r_0^4\right)
\end{align}
The result now follows from the elementary estimate
$$
|\det H_f(x) -\det H_q| < 4\norm{H_q}\omega + 2\omega^2.
$$
for $x \in \interior T$.
\end{proof}

\subsection{A crude bound for PL approximations over fat triangulations} \label{sect: fat}
Put $\Theta(T)$ for the smallest angle of any triangle of $T$. Thus $\Theta(T)$ is a measure of the ``fatness" of $T$. We put also $M(T)$ for the length of the longest edge of $T$, and $m(t)$ for the length of the shortest edge.

\begin{proposition}\label{prop:crude} Given $\Theta \in (0,\frac \pi 3)$ and $k \in (0,1)$  there exists a constant $C=C (\Theta,k)$ with the following property. Let $T$ be a finite triangulation of an open polygon $W\subset \R^2$, with
$$
\Theta (T )\ge \Theta, \quad \frac{m(T)}{M(T)} \ge k.
$$ 
Then for any $f \in C^2(W)$ 
$$
\mass~ \D(f_T) \le C \area(W)\left( 1 + \sup_W{\norm{H_f}}+  \sup_W{\norm {H_f}}^2\right)
$$
\end{proposition}
\begin{proof} We decompose $\D(f_T) = D_0 + D_1 + D_2$ as above. Clearly 
\begin{equation}\label{eq:mass 2}
\mass~D_2 = \area W
\end{equation}

 By the mean value theorem, the gradient  $\nabla_\tau$ of $f_T$ on a given triangle $\tau \in T$ lies in the convex hull of the image of $\tau$ under $df$. It follows that if the triangles $\tau_1, \tau_2\in T$ share an edge or a vertex then 
 \begin{equation}\label{eq:diam}
 |\nabla_{\tau_1} - \nabla_{\tau_2}| \le 2 \sup {\norm{H_f} }M( T)
 \end{equation}
   Therefore by the construction section of \ref{sect:d of pl}
\begin{equation*}
\mass~ D_1 \le 2\sup \norm{H_f} M(T) \sum _j\length(\sigma_j)  
\end{equation*}
where the sum is taken over all edges $\sigma_j$ of $T$. Since clearly
$$
m(T) \sum \length (\sigma_j) \le\frac 6{\sin \Theta} \area W
$$
we conclude that 
\begin{equation}\label{eq:mass 1}
\mass~ D_1 \le \frac{12}{\sin \Theta} \frac{M}{m}  \area W \sup \norm{H_f}
\end{equation}

Clearly the number of triangles (or edges) incident to any given vertex $\rho$ of $T$ is no greater than $C':=\frac{2\pi}{\Theta(T)}$. By the construction of section \ref{sect:d of pl}, the fiber of $D_0$ over $\rho$ is a polygon $P$ with multiplicities and with at most $C'$ edges. It follows that these multiplicities are at most $\frac {C'} 2= \frac \pi{\Theta(T)}$, so that by \eqref{eq:diam}
\begin{equation*}
\mass~D_0 \le \frac {4\pi^2}{\Theta(T)}   \sup{\norm{H_f}}^2 M(T)^2 V
\end{equation*}
where $V$ is the number of  vertices in $T$. Putting $E$ for the number of  edges of $T$, clearly $2E \ge   V $. Therefore
$$
\frac{m(T)^2} 2 V\le m(T)^2 E \le \frac 6 {\sin\Theta} \area W
$$
and we obtain
\begin{equation}\label{eq:mass 0}
\mass~ D_0 \le \frac{48\pi^2}{\Theta \sin \Theta}\frac{M(T)^2}{m(T)^2} \area W \sup \norm{H_f} ^2
\end{equation}
The conclusion now follows if we take $C:=\frac{48\pi^2}{k^2\Theta \sin \Theta}$.
\end{proof}

\section{Proof of Theorem \ref{thm:main}}
Givem $n \in \mathbb N$ we put $W_n$ for the open polygon consisting of the interior of the union of all closed squares $S = S_{a,b}= [\frac a n,\frac {a+1} n]\times [\frac b n,\frac {b+1} n], \ a,b \in \Z$, that intersect $V$. It is clear that $V\subset W_n \subset U$ for $n$ sufficiently large, and that 
\begin{equation}\label{eq:areas}
\lim_{n\to \infty}\area W_n = \area V
\end{equation} 

For $n,j \in \mathbb N$, consider the triangulation $T_{n,j} $ of $W_n$ constructed as follows. Given a  square $S$ of $W_n$, let $T_{S,j}$ be a square mesh within $S$ that 
\begin{itemize}
\item has  size $\frac 1{nj}$ 
\item has axes parallel to the eigenvectors of $H_f(x)$, where  $x$ is the center of $S$, and
\item is maximal with respect to the condition that no vertex of $T_S$ lies within distance $\frac 1 {nj}$ of the complement of $S$
\end{itemize} Let $T_{n,j}'$ be the system of vertices and edges in $W_n$ consisting of all  vertices and edges of the $T_{S,j}$, together with the edges and vertices obtained by subdividing each edge of every square $S$ into $j$ equal segments of length $\frac 1 {nj}$. This system satisfies the conditions of Theorem \ref{thm:chew}, and we take $T_{n,j}$ to be the triangulation of $W_n$ in the conclusion of that statement.

Referring to Definition \ref{def:complexes}, put  $V_{n,j}:= \bigcup_S \core {T_{S,j}}$. Put $Z_{n,j}$ for the interior of the union of $\bigcup_S \crust T_{S,j}$ with all of the additional triangles in $T_{n,j}$ that were obtained by means of Theorem \ref{thm:chew}. Thus $V_{n,j} \cup Z_{n,j} =W_n \supset  V$. Furthermore, each $S-S\cap Z_{n,j}$ includes a square of side $\frac 1 n(1 - \frac{4\sqrt 2}{j})$. It follows that
 \begin{equation}
\area(Z_{n,j}) \le \frac{8\sqrt 2}{j}\area(W_n).
 \end{equation}

 Now put $p_{n,j}:= f_{T_{n,j}}$. It is clear that $p_{n,j} \to f$ locally uniformly as $ n,j \to \infty$.
 We will show that there are $n_1< n_2< \dots $ such that the sequence PL functions $p_{n_j,j}$ satisfies the relation \eqref{eq:main} in the statement of Theorem \ref{thm:main}.

Let $U'$ be another open set such that $V\csubset U'\csubset U$. We suppose from now on that $n$ is large enough that
 $W_n\subset U'$.

 In view of \eqref{eq:areas}, it follows from Proposition \ref{prop:crude} that
\begin{equation}\label{eq:small}
\mass(\D(\restrict{p_{n,j}}{Z_{n,j}}) )\le C \area(Z_{n,j}) (1 + \sup_{U'}\norm{H_f} + \sup_{U'}\norm{H_f}^2)\to 0
\end{equation}
as $n,j \to \infty$,
where $C= C(\frac \pi 6, \frac 1 2)$.

\begin{figure}[h]
\begin{center}
\setlength{\unitlength}{150pt}
\begin{picture}(1,1)
\put(0,0){\includegraphics[width=\unitlength]{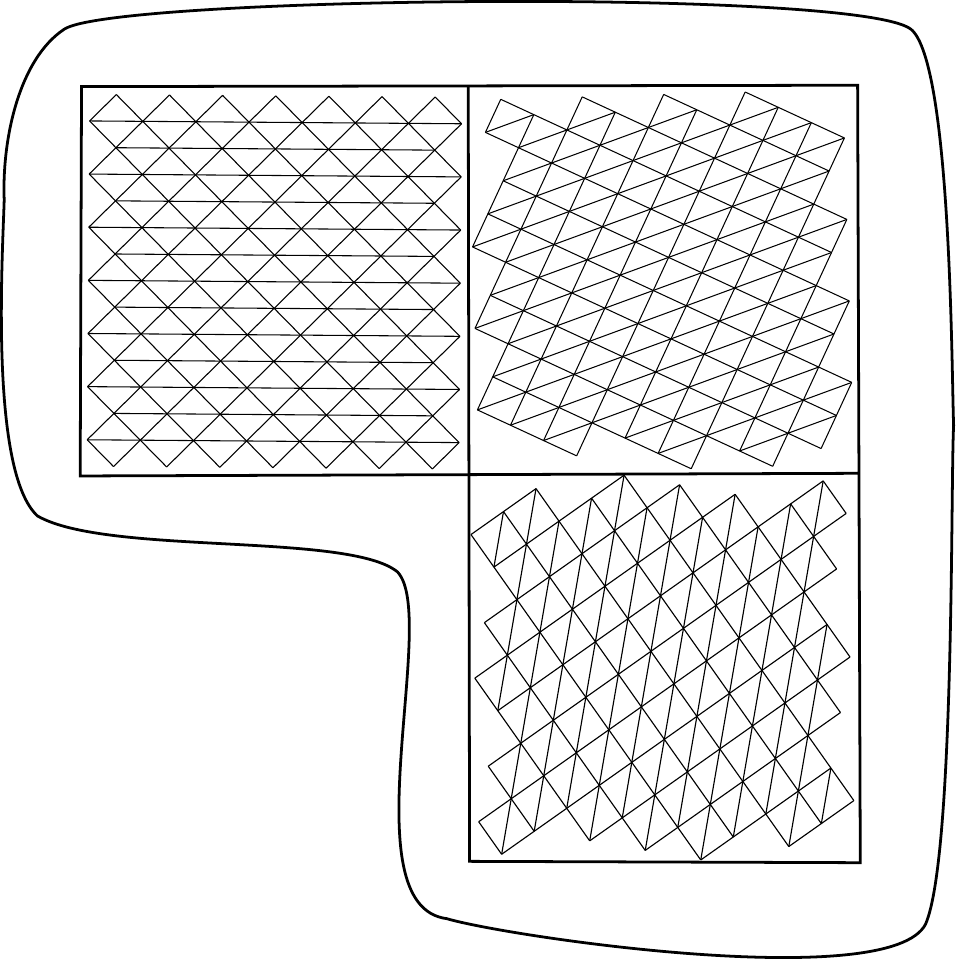}}
\end{picture}
\end{center}
\caption{A triangulation for most of the region} 
\label{fig:region}
\end{figure}

Since $V\subset V_{n,j} \cup Z_{n,j}$,
\begin{equation*}
\mass(\D(\restrict{p_{n,j}}{V})) \le \mass(\D(\restrict{p_{n,j}}{V_{n,j}})) +\mass(\D(\restrict{p_{n,j}}{Z_{n,j}}))
\end{equation*}
Thus by \eqref{eq:small} we may complete the proof by showing that for fixed $j$
\begin{equation}
\label{eq:big term}\limsup_{n\to \infty} \mass(\D(\restrict{p_{n,j}}{V_{n,j}}))< \int_{V} 1 + 2\sqrt 2 \norm{ H_f} + |\det H_f | \, dA
\end{equation}

Put $L:= \sup_{U'}\norm{H_f}$. Let $\omega$ be a modulus of continuity for $\restrict{H_f}{U'}$, i.e. a monotonically increasing function with $\lim_{r\downarrow 0} \omega(r) = 0$, such that
$$
x,y \in U', |x-y| < r \implies \norm{H_f(x)-H_f(y)} < \omega(r).
$$
Applying Lemma \ref{lem:taylor} with $r_0 = \frac 1{n\sqrt 2}, \omega = \omega_n:=\omega\left(\frac 1{n\sqrt 2}\right)$ and $\eps = \frac 1 {nj}$, we obtain
\begin{align*}\label{eq:square estimate}
\mass (\D(\restrict {p_{n,j}}{V_{n,j}}) )&< \int_{V_{n,j}} 1 + 2\sqrt 2 \norm{H_f} + \left|\det H_f\right| \, dA \\
\notag&+\area(V_{n,j})\left[(6 \sqrt 2 + 12)j^2 \omega_n + 6 j^4 \omega_n^2 + (2\sqrt 2 + 2L)\omega_n + 2\omega_n^2 \right]
\end{align*}
The dominated convergence theorem implies that the integral converges to the corresponding integral over $V$ as $n\to \infty$, while the second term converges to zero.
This concludes the proof.

\end{document}